\documentclass[a4paper,12pt]{amsart}

\usepackage[T1]{fontenc}
\usepackage[utf8]{inputenc}
\usepackage{lmodern,amsmath,amssymb,array,tabularx,enumitem,microtype}
\usepackage{color}
\usepackage{placeins}
\usepackage[left=3cm,right=3cm]{geometry}

\usepackage[hidelinks]{hyperref}
\usepackage{xurl}
\usepackage[normalem]{ulem}

\makeatletter
\def\thm@space@setup{%
  \thm@preskip=14pt plus 3pt minus 2pt
  \thm@postskip=14pt plus 3pt minus 2pt
}
\def\section{\@startsection{section}{1}%
  \z@{.7\linespacing\@plus\linespacing}{14pt plus 2pt minus 1pt}%
  {\normalfont\scshape\centering}}
\makeatother
\newtheorem{thm}{Theorem}[section]
\newtheorem{proposition}[thm]{Proposition}
\newtheorem{lemma}[thm]{Lemma}
\newtheorem{corollary}[thm]{Corollary}
\theoremstyle{definition} 
\newtheorem{remark}[thm]{Remark}
\theoremstyle{plain} 
\newtheorem{defn}[thm]{Definition}
\newtheorem{qn}[thm]{Question}
\theoremstyle{definition}
\newtheorem{example}[thm]{Example}
\theoremstyle{plain}

\newtheorem*{theorem*}{Theorem}
\newtheorem*{corollary*}{Corollary}
\newtheorem*{proposition*}{Proposition}
\newtheorem*{question*}{Question}
\newtheorem*{lemma*}{Lemma}
\newtheorem*{problem*}{Problem}
\newtheorem*{definition*}{Definition}

\numberwithin{equation}{section}

\newcommand{\Irr}{\operatorname{Irr}}
\newcommand{\Cl}{\operatorname{Cl}}
\DeclareMathOperator{\F}{\mathbb F}
\DeclareMathOperator{\GL}{\mathrm{GL}}
\DeclareMathOperator{\Z}{\mathbb Z}
\DeclareMathOperator{\Q}{\mathbb Q}
\DeclareMathOperator{\R}{\mathbb R}
\DeclareMathOperator{\C}{\mathbb C}
\DeclareMathOperator{\cv}{cv}

\begin{document}

\title{Dimension statistics of representations of finite groups}
\author{Arvind Ayyer}
\address{Arvind Ayyer, Department of Mathematics,
  Indian Institute of Science, Bangalore  560012, India.}
\email{arvind@iisc.ac.in}

\author{Dipendra Prasad}
\address{Dipendra Prasad, Department of Mathematics, 
IIT Bombay, Powai, Mumbai 400076, India}
\email{prasad.dipendra@gmail.com}

\date{\today}

\begin{abstract}
The paper studies the distribution of dimensions of irreducible representations of a finite group and compares them with the sizes of its conjugacy classes. We study these questions for a fixed reductive group over growing finite fields, for $\GL_n(\F_q)$ where $q$ is fixed but $n$ varies, and for the symmetric group $S_n$ as $n$ varies. We also discuss finite nilpotent groups arising from Lie algebras over finite fields. As a consequence of Kirillov theory, we deduce that for ``selfdual'' nilpotent groups whose nilpotency class is less than the characteristic of the field, the squares of the dimensions of irreducible representations match the sizes of conjugacy classes, with multiplicity. These groups thus provide a large class of examples with this property.
\end{abstract}


\maketitle

\setcounter{tocdepth}{1}
\tableofcontents
\clearpage

\section{Introduction: exact comparisons and asymptotic substitutes}
\label{sec:intro}
All representations in this paper are finite-dimensional complex representations. For a finite group $H$, let $\Cl(H)$ denote the set of conjugacy classes in $H$, and $\operatorname{Irr}(H)$ the set of isomorphism classes of irreducible representations of $H$. Write $k(H)=|\operatorname{Irr}(H)|=|\Cl(H)|$. If $d_\pi=\dim\pi$, the dimension identity and the class equation are
\begin{equation}
\label{eq:1.1}
\sum_{\pi\in\operatorname{Irr}(H)}d_\pi^2=|H|,\qquad\sum_{C\in\Cl(H)}|C|=|H|.\end{equation}
These are two partitions of the same integer with the same number of summands. The motivating question for this paper is: to what extent can they be compared?

Since there is no natural correspondence between the irreducible representations of a finite group and its conjugacy classes, a natural way to compare the two partitions is to write their parts in increasing order, $d_1^2\leq\cdots\leq d_k^2$ and $c_1\leq\cdots\leq c_k$, and then compare them term-by-term or in a statistical or average sense, as the group grows in a systematic way.

There is a useful and natural relaxation of exact multiset equality. One may replace conjugacy classes by orbits of another group $H'$, with $|H|=|H'|$, acting on a set of cardinality $|H|$, as in the coadjoint orbit method (where $H=H'$, but the action is the coadjoint action on the dual Lie algebra, rather than conjugation on the group itself); or, for reductive groups over finite fields, one may take $H'$ to be the group of rational points of the dual reductive group, acting on itself by conjugation.

For later use, we distinguish three vectors:
\[
d_H=(d_\pi)_{\pi\in\operatorname{Irr}(H)},\qquad s_H=(\sqrt{|C|})_{C\in\Cl(H)},\qquad c_H=(|C|)_{C\in\Cl(H)}.
\]
The first two have $L^2$-norm $\sqrt{|H|}$; the third has no comparable general norm formula. Note that an ordering is needed to compare $d_H$ with $s_H$ in calculating $\|d_H-s_H\|$, whereas no ordering is needed to compare either vector with $\mathbf1$, the vector whose coordinates are all $1$ and whose size is the same as that of $d_H$.

\subsection{The three asymptotic regimes} 
For a nonzero vector $a$ with nonnegative entries in the Euclidean space $\mathbb R^k$, let $\theta(a)=\angle(a,\mathbf1)$ be the angle between $a$ and the vector $\mathbf{1}$, all of whose components are $1$ and which has the same size as $a$. 
We define, for an integer $q\geq2$, the quantity $\gamma(q)=\sum_{j\geq0}q^{-j(j+1)/2}$.
The main angle conclusions that we draw in this paper are shown in Table~\ref{tab:main}.

The angle measures the dispersion of each data set: $\theta\to0$, or equivalently $\cos^2\theta\to1$, means that the corresponding vector is asymptotically constant. When this holds for both $d_H$ and $s_H$, their normalized distance $\|d_H-s_H\|/\sqrt{|H|}$ tends to zero under every pairing, giving a precise sense in which the two data sets are statistically the same. As Table~\ref{tab:main} shows, this happens for a fixed reductive group $G$ when $q\to\infty$. For $\GL_n(\F_q)$ with $q$ fixed and $n\to\infty$, the squared cosine for the dimension vector tends to a limit strictly between zero and one. For $S_n$, the dimension vector and the vector of conjugacy-class sizes both become asymptotically orthogonal to the constant vector as $n\to\infty$.

\begin{table}[h!]
\begin{center}
\renewcommand{\arraystretch}{1.4}
\begin{tabular}{@{}lll@{}}\hline
Family and limit & Vector(s) &Limit of $\cos^2\theta$\\\hline
Fixed reductive $G$, $q\to\infty$&$d_{G(\F_q)}$ and $s_{G(\F_q)}$&$1$\\
$\GL_n(\F_q)$, fixed $q$, $n\to\infty$&$d_{\GL_n(\F_q)}$&$1/\gamma(q)\in(0,1)$\\
$S_n$, $n\to\infty$&$d_{S_n}$&$0$\\
$S_n$, $n\to\infty$&$c_{S_n}$ &$0$\\\hline
\end{tabular}
\medskip
\caption{The angles for limiting vectors for different families of groups.}
\label{tab:main}
\end{center}
\end{table}
\FloatBarrier

Section~\ref{sec:angles} makes the statistical meaning of the angle explicit. The logarithmic condition introduced there holds in all three dimension regimes, but is substantially weaker.
The first row is discussed in Section~\ref{sec:reductive}.
The second row is discussed in Section~\ref{sec:gl}.
The fixed-$q$ calculation uses the established formulas of Gow~\cite{Gow83}, 
Feit--Fine~\cite{FF60}
and Macdonald~\cite{Mac81}, and
Fulman--Guralnick~\cite{FG12}. 
The third row is discussed in Section~\ref{sec:dimensions} and the fourth row, in 
Section~\ref{sec:classes}.

Section~\ref{sec:orbit} summarizes Kirillov theory for finite nilpotent groups arising from Lie algebras of nilpotency class less than the characteristic. This theory is then used in Section~\ref{sec:selfdual} to construct explicit examples of finite nilpotent groups for which the squared dimensions of irreducible representations match the sizes of conjugacy classes, with multiplicity. That section discusses \textit{selfdual} nilpotent groups and a refinement of this concept, which we call \textit{filtered selfdual} nilpotent groups.

For a fixed connected reductive group over growing finite fields, Section~\ref{sec:reductive} shows that almost all irreducible dimensions are close to the common scale $q^N$, where $N=(\dim G-\operatorname{rank}G)/2$. For $S_n$, Section~\ref{sec:dimensions} shows that the proportion of irreducibles whose dimensions exceed any fixed positive fraction of the maximal dimension tends to zero.

\subsection{Examples against exact equality} 
For $S_3$ and $S_4$, the two identities in \eqref{eq:1.1} read
\begin{align*}
6&=1+1+4=1+2+3,\\
24&=1+1+4+9+9=1+3+6+6+8.
\end{align*}
For either the quaternion group or the dihedral group of order eight, they read
\[
8=1+1+1+1+4=1+1+2+2+2.
\]
Thus even nilpotence does not guarantee the desired equality. 
Nevertheless, there are nonabelian examples with exact equality, as we will discuss later in the paper.

\subsection{Examples in favor and a low-rank comparison} For an abelian group all dimensions and all class sizes equal one. For $\GL_2(\F_q)$, the classification of irreducible representations and that of conjugacy classes give the following comparison:

\begingroup
\setlength{\topsep}{18pt plus 3pt minus 2pt}
\begin{center}
\renewcommand{\arraystretch}{1.3}
\begin{tabular}{cc|cc}
\hline
Dimension&Multiplicity&Class size&Multiplicity\\
\hline
$1$&$q-1$&$1$&$q-1$\\
$q$&$q-1$&$q^2-1$&$q-1$\\
$q+1$&$(q-1)(q-2)/2$&$q(q+1)$&$(q-1)(q-2)/2$\\
$q-1$&$q(q-1)/2$&$q(q-1)$&$q(q-1)/2$ \\
\hline
\end{tabular}
\end{center}
\endgroup

The irreducible representations in the four rows above are one-dimensional representations, twists of the Steinberg representation by one-dimensional representations, principal series representations, and cuspidal representations, respectively. The conjugacy classes in the four rows above are scalar, nonsemisimple with one eigenvalue, split regular semisimple, and nonsplit regular semisimple, respectively. The multiplicities sum to $q^2-1$ on either side. This table makes a multiset asymptotic comparison possible, and although the squared dimensions and class sizes are not identical, as polynomials in $q$ they have the same degree and leading terms.

For $\mathrm{SL}_2(\F_q)$ with odd $q\geq5$, the irreducible representations have the following dimensions and multiplicities:
\begingroup
\setlength{\abovedisplayskip}{18pt plus 3pt minus 2pt}
\setlength{\belowdisplayskip}{18pt plus 3pt minus 2pt}
\setlength{\abovedisplayshortskip}{18pt plus 3pt minus 2pt}
\setlength{\belowdisplayshortskip}{18pt plus 3pt minus 2pt}
\[
\renewcommand{\arraystretch}{1.45}
\begin{array}{c|cccccc}
\text{dimension}&1&q&q+1&q-1&(q+1)/2&(q-1)/2\\\hline
\text{multiplicity}&1&1&(q-3)/2&(q-1)/2&2&2
\end{array}
\]
\endgroup
In $\mathrm{PGL}_2(\F_q)$, the split and nonsplit involution classes have sizes $q(q+1)/2$ and $q(q-1)/2$. Their centralizers are the respective torus normalizers, accounting for the factor two in these class-size formulas. The leading terms therefore agree with twice the squares of the exceptional dimensions $(q\pm1)/2$ of certain representations of $\mathrm{SL}_2(\F_q)$. Thus, we have a similar phenomenon for representations of $\mathrm{SL}_2(\F_q)$, as was observed for $\GL_2(\F_q)$, relating squared dimensions and class sizes treated as polynomials in $q$, but with dimensions of representations of $\mathrm{SL}_2(\F_q)$ compared with conjugacy-class sizes in $\mathrm{PGL}_2(\F_q)$, the group of rational points of its dual reductive group.

\section{Finite nilpotent groups: the orbit method}
\label{sec:orbit}
The purpose of this section is to recall Kirillov theory about representations of finite nilpotent groups. We begin by recalling the correspondence between the central series of a nilpotent Lie algebra and those of the corresponding group. In characteristic zero, let $\mathfrak n$ be a finite-dimensional nilpotent Lie algebra and let $N=\exp(\mathfrak n)$ be the corresponding connected simply connected nilpotent group, with multiplication given by the Baker--Campbell--Hausdorff formula. Write
\[
\mathfrak n=\gamma_1\supset\gamma_2\supset\cdots,
\qquad \gamma_{i+1}=[\mathfrak n,\gamma_i],
\]
and
\[
N=N_1\supset N_2\supset\cdots,
\qquad N_{i+1}=[N,N_i].
\]
Then $N_i=\exp(\gamma_i)$ for every $i$; in particular,
\[
[N,N]=\exp([\mathfrak n,\mathfrak n]).
\]
Likewise, if
\[
0=\mathfrak z_0\subset\mathfrak z_1\subset\cdots,
\qquad \mathfrak z_{i+1}/\mathfrak z_i=Z(\mathfrak n/\mathfrak z_i),
\]
is the upper central series of $\mathfrak n$, and
\[
1=N^0\subset N^1\subset\cdots,
\qquad N^{i+1}/N^i=Z(N/N^i),
\]
is that of $N$, then $N^i=\exp(\mathfrak z_i)$ for every $i$. Thus the entire pair of lower and upper central filtrations passes unchanged between the Lie algebra and the simply connected nilpotent group.

Over a finite field $\F_q$ of characteristic $p$, the analogous statement holds for a finite-dimensional nilpotent Lie algebra $\mathfrak n$ in the Lazard range, i.e., when its nilpotency class is strictly less than $p$. The group product is given by the truncated Baker--Campbell--Hausdorff formula (identifying $x\in\mathfrak n$ with $\exp(x)\in N$):
\[
x*y=x+y+\frac12[x,y]
+\frac1{12}[x,[x,y]]-\frac1{12}[y,[x,y]]+\cdots.
\]
All denominators are invertible in characteristic $p$. This gives a finite $p$-group $N=\exp(\mathfrak n)$ on the underlying set of $\mathfrak n$. The maps $\exp$ and $\log$ are inverse polynomial maps, and one again has
\[
N_i=\exp(\gamma_i),\qquad N^i=\exp(\mathfrak z_i)
\]
at every level. In particular,
\[
Z(N)=\exp(Z(\mathfrak n)),\qquad
[N,N]=\exp([\mathfrak n,\mathfrak n]),
\]
and the same correspondence holds after taking the successive relevant quotients. Hence a condition comparing the lower and upper central filtrations on the Lie-algebra side has the corresponding interpretation on the group side.

The Lazard correspondence also identifies conjugacy classes in $N$ with adjoint orbits in $\mathfrak n$. For $x,y\in\mathfrak n$, one has
\[
\exp(x)\exp(y)\exp(-x)=\exp\!\left(e^{\operatorname{ad}x}y\right).
\]
Consequently, exponentiation restricts to a bijection
\[
\operatorname{Ad}(N)y\ \longrightarrow\
\operatorname{Cl}_N(\exp y),
\qquad z\longmapsto\exp z.
\]
The corresponding centralizers are
\[
C_N(\exp y)=\exp(C_{\mathfrak n}(y)),
\qquad
C_{\mathfrak n}(y)=\{x\in\mathfrak n:[x,y]=0\}.
\]
Thus the adjoint orbit and conjugacy class have the same cardinality,
\[
|\operatorname{Ad}(N)y|
=|\operatorname{Cl}_N(\exp y)|
=\frac{|N|}{|C_N(\exp y)|}
=q^{\dim\mathfrak n-\dim C_{\mathfrak n}(y)}.
\]

We now discuss the coadjoint comparison in this finite-field setting. We use the same notation $N=\exp(\mathfrak n)$ for the associated algebraic unipotent group and specify rational points when needed.

Fix a nontrivial additive character $\psi:\F_q\to\C^\times$. The map $\lambda\mapsto\psi\circ\lambda$ identifies $\mathfrak n^*=\operatorname{Hom}_{\F_q}(\mathfrak n,\F_q)$ with the additive-character dual of $\mathfrak n$. This identification is bijective.

For $\lambda\in\mathfrak n^*$, set $B_\lambda(x,y)=\lambda([x,y])$ and
\[
\mathfrak n_\lambda=\{x \mid B_\lambda(x,y)=0\text{ for every }y\in\mathfrak n\}.
\]
The alternating form $B_\lambda$ induces a nondegenerate symplectic form on $\mathfrak n/\mathfrak n_\lambda$. A polarization for the pair $(\mathfrak n,\lambda)$ is a Lie subalgebra $\mathfrak h_\lambda\subseteq\mathfrak n$ containing $\mathfrak n_\lambda$, with $\lambda([\mathfrak h_\lambda,\mathfrak h_\lambda])=0$, and of dimension
\[
\dim\mathfrak h_\lambda=\tfrac12(\dim\mathfrak n+\dim\mathfrak n_\lambda).
\]
Polarizations exist in this nilpotent Lie setting, by the standard polarization construction along a complete flag of ideals.

The following form of the Kirillov orbit method is recalled in O'Brien--Voll~\cite[Theorem 3.1]{OV15}; see also Boyarchenko--Sabitova~\cite[Theorem 2.6]{BS08}.

\begin{thm}
\label{thm:kirillov}
In the preceding nilpotence-class-less-than-$p$ setting, 
set
$H_\lambda=\exp(\mathfrak h_\lambda)$ and 
$\vartheta_\lambda(h)=\psi\bigl(\lambda(\log h)\bigr)$ for $h\in H_\lambda(\F_q)$.
Then $\vartheta_\lambda$ is a linear character of $H_\lambda(\F_q)$ and
\[
\pi_\lambda=\operatorname{Ind}_{H_\lambda(\F_q)}^{N(\F_q)}\vartheta_\lambda
\]
is irreducible. Its isomorphism class is independent of the polarization and depends precisely on the coadjoint orbit of $\lambda$, giving a bijection of $\operatorname{Irr}(N(\F_q))$ with coadjoint orbits of $N(\F_q)$ on $\mathfrak n^*$, with
\[
\dim\pi_\lambda=[N(\F_q):H_\lambda(\F_q)]=\sqrt{|N(\F_q)\cdot\lambda|}.
\]
\end{thm}

To explain the square root, note that the stabilizer of $\lambda$ in $N(\F_q)$ is $\exp(\mathfrak n_\lambda)$ and has order $q^{\dim\mathfrak n_\lambda}$. Hence
\[
|N(\F_q)\cdot\lambda|=q^{\dim\mathfrak n-\dim\mathfrak n_\lambda}=[N(\F_q):H_\lambda(\F_q)]^2.
\]

\begin{remark}
\label{rem:orbit-failure}  
Without further hypothesis (such as the preceding nilpotence-class-less-than-$p$ setting), the equality of squared dimensions with coadjoint orbit sizes is false even for algebra groups $1+J$, where $J$ is a finite-dimensional nilpotent associative algebra. This is the \textit{fake degree conjecture}. Jaikin--Zapirain gave a counterexample in characteristic two \cite{JZ04}; Garc\'ia-Rodr\'iguez, Jaikin--Zapirain and Jezernik gave counterexamples in every characteristic \cite[Corollary 4]{GJJ17}. Here the discrepancy is already visible in dimension one.
\end{remark}

\section{Selfdual nilpotent Lie algebras}
\label{sec:selfdual}
Under the hypotheses of Theorem~\ref{thm:kirillov}, coadjoint orbits describe the dimension data for irreducible representations of finite nilpotent groups. To be able to relate the dimension data of such nilpotent groups with conjugacy classes, we therefore begin by discussing when a nilpotent Lie algebra is selfdual, so that its adjoint and coadjoint representations can be identified and, by Theorem~\ref{thm:kirillov}, the squares of the dimensions of irreducible representations are given by the sizes of orbits under the adjoint action instead of the coadjoint action.

In analysing selfduality for nilpotent Lie algebras, we come across a stronger condition than selfduality, which we call \textit{filtered selfduality}. Such Lie algebras are often encountered in practice and have pleasanter properties.

\subsection{Selfduality and the central series of a nilpotent Lie algebra}
Let $\mathfrak g$ be a nilpotent Lie algebra of class $c$. Use the standard indexing for its lower central series:
\[
\mathfrak g=\gamma_1\supset\gamma_2\supset\cdots\supset\gamma_c\supset\gamma_{c+1}=0,
\qquad \gamma_{i+1}=[\mathfrak g,\gamma_i].
\]
Write its upper central series as
\[
0=\mathfrak z_0\subset \mathfrak z_1\subset\cdots\subset \mathfrak z_c=\mathfrak g,
\qquad \mathfrak z_{i+1}/\mathfrak z_i=Z(\mathfrak g/\mathfrak z_i).
\]
Equivalently, $\mathfrak z_{i+1}=\{x\in\mathfrak g:[\mathfrak g,x]\subseteq \mathfrak z_i\}$, with $\mathfrak z_0=0$.
An isomorphism of $\mathfrak g$-modules
\[
\operatorname{Ad}\simeq\operatorname{Ad}^{*},
\qquad\text{equivalently}\qquad
\mathfrak g\simeq\mathfrak g^{*},
\]
is equivalent to the existence of a nondegenerate invariant bilinear pairing $B:\mathfrak g\times\mathfrak g\to k$ satisfying
\[
B([x,y],z)+B(y,[x,z])=0.
\]
For a subspace $U\subseteq\mathfrak g$, we use the convention
\[
U^{\perp_B}=\{x\in\mathfrak g:B(u,x)=0\text{ for every }u\in U\}.
\]

\medskip\noindent\textit{Orthogonality of the two central series.}
Invariance of $B$ relates the lower and upper central series by
\begin{equation}
\label{eq:central-series-orthogonality}
\gamma_i^{\perp_B}=\mathfrak z_{i-1},\qquad 1\leq i\leq c+1.
\end{equation}
Indeed, $\gamma_1^{\perp_B}=\mathfrak g^{\perp_B}=0=\mathfrak z_0$ and $\gamma_2^{\perp_B}=[\mathfrak g,\mathfrak g]^{\perp_B}=Z(\mathfrak g)=\mathfrak z_1$. If $\gamma_{i-1}^{\perp_B}=\mathfrak z_{i-2}$, then invariance and nondegeneracy give
\begin{align*}
x\in\gamma_i^{\perp_B}
&\Longleftrightarrow B([\mathfrak g,\gamma_{i-1}],x)=0\\
&\Longleftrightarrow B(\gamma_{i-1},[\mathfrak g,x])=0\\
&\Longleftrightarrow [\mathfrak g,x]\subseteq \mathfrak z_{i-2}\\
&\Longleftrightarrow x\in \mathfrak z_{i-1}.
\end{align*}
Therefore the invariant form $\langle\, ,\,\rangle$ on $\mathfrak g$ pairs each lower-series quotient $\gamma_i/\gamma_{i+1}$ perfectly with the upper-series quotient $\mathfrak z_i/\mathfrak z_{i-1}$. Invariance thus canonically identifies the orthogonal filtration of the lower central series with the upper central series. There is a prettier and often encountered special case, obtained by imposing the stronger condition that the lower central filtration be selfdual under degree reversal. This condition induces perfect invariant pairings between complementary degrees of the associated graded Lie algebra. We give this strengthening its own name.

\begin{defn}[Filtered selfdual nilpotent Lie algebra]
\label{def:filtered-selfdual-lie}
A nilpotent Lie algebra $\mathfrak g$ of class $c$ is \textit{filtered selfdual} if it admits a nondegenerate invariant \emph{symmetric} bilinear form $B$ such that
\[
\gamma_i^{\perp_B}=\gamma_{c+2-i}\qquad(1\leq i\leq c+1).
\]
Such a form will be called a \textit{selfdualizing form}. By~\eqref{eq:central-series-orthogonality}, for an invariant nondegenerate form this filtration condition is equivalent to
\[
\mathfrak z_i=\gamma_{c+1-i}\qquad(0\leq i\leq c),
\]
that is, agreement of the upper and lower central series in reverse order.
\end{defn}

\medskip\noindent\textbf{The induced duality on the associated graded algebra.}
Put $\mathfrak g_i=\gamma_i/\gamma_{i+1}$ for $1\leq i\leq c$. Then
\[
\operatorname{gr}(\mathfrak g)=\bigoplus_{i=1}^{c}\mathfrak g_i,
\qquad [\mathfrak g_i,\mathfrak g_j]\subseteq \mathfrak g_{i+j},
\]
where $\mathfrak g_r=0$ for $r>c$. The pairing $\langle\ ,\ \rangle:\mathfrak g\times\mathfrak g\to k$ given by $B$ induces perfect pairings
\[
\mathfrak g_i\times \mathfrak g_{c+1-i}\longrightarrow k,
\qquad 1\leq i\leq c.
\]
This is the graded consequence of Definition~\ref{def:filtered-selfdual-lie}: the lower central filtration is selfdual under $B$, and its opposite graded pieces pair perfectly.

\medskip\noindent\textbf{The induced alternating trilinear form.}
For $i+j+k=c+1$, the graded bracket and opposite-degree pairing define
\[
T_{ijk}:\mathfrak g_i\times \mathfrak g_j\times \mathfrak g_k\longrightarrow k,
\qquad T_{ijk}(x,y,z)=\langle x,[y,z]\rangle.
\]
Here $[y,z]\in \mathfrak g_{j+k}=\mathfrak g_{c+1-i}$. The invariance and symmetry of $B$ give
\[
T_{ijk}(x,y,z)=T_{jki}(y,z,x)=T_{kij}(z,x,y),
\]
while skew-symmetry of the bracket gives
$T_{ijk}(x,y,z)=-T_{ikj}(x,z,y)$. Hence these components form an alternating trilinear form on $\operatorname{gr}(\mathfrak g)$, supported in total degree $c+1$.

\begin{proposition}[Class-two criteria for selfduality and filtered selfduality]
\label{prop:class-two-filtered-selfdual}
Let $\mathfrak g$ be a finite-dimensional nilpotent Lie algebra of class two over a field $k$ of characteristic different from $2$. Put
\[
D=[\mathfrak g,\mathfrak g],\qquad C=Z(\mathfrak g),
\]
so that $D\subseteq C$.

\textbf{(i) Selfduality.}
Put $V=\mathfrak g/C$, and let
\[
\beta:\bigwedge^2V\longrightarrow D
\]
be the surjective map induced by the bracket. (A surjective map $\beta$ defines a unique class-two Lie bracket on $\mathfrak g$ by $[x,y]=\beta(x+C,y+C)$. Its center is exactly $C$ if and only if $\{v\in V:\beta(v,V)=0\}=0$, as holds for the induced bracket here.) Then $\mathfrak g$ admits a nondegenerate invariant bilinear form (which we do not assume to be symmetric) if and only if there is a perfect pairing
\[
\Phi:V\times D\longrightarrow k
\]
such that
\[
T(x,y,z)=\Phi\bigl(x,\beta(y,z)\bigr)
\]
is alternating. Whenever this condition holds, the invariant form can be chosen symmetric.

\textbf{(ii) Filtered selfduality.}
The algebra $\mathfrak g$ is filtered selfdual if and only if it is selfdual and $C=D$. Equivalently, putting $U=\mathfrak g/D$ and denoting the induced bracket by
\[
\widetilde\beta:\bigwedge^2U\longrightarrow D,
\]
there is a perfect pairing
\[
\widetilde\Phi:U\times D\longrightarrow k
\]
for which
\[
\widetilde T(x,y,z)
=\widetilde\Phi\bigl(x,\widetilde\beta(y,z)\bigr)
\]
is alternating.

Moreover, every selfdual $\mathfrak g$ admits a decomposition
\[
\mathfrak g=\mathfrak h\oplus\mathfrak a
\]
into ideals, where $\mathfrak h$ is filtered selfdual of class two and $\mathfrak a$ is central and abelian, with
\[
\dim\mathfrak a=\dim C-\dim D.
\]
For any nondegenerate invariant bilinear form $B$ on $\mathfrak g$,
\[
D^{\perp_B}=C,\qquad C^{\perp_B}=D.
\]
In particular, $D$ is totally isotropic, and $B$ induces a nondegenerate bilinear form on $C/D$.
\end{proposition}

\begin{proof}
Suppose that $B$ is a nondegenerate invariant bilinear form. Invariance gives
\[
B(x,[y,z])=B([x,y],z).
\]
Hence
\[
B(x,D)=0
\iff [x,\mathfrak g]=0
\iff x\in C.
\]
The analogous identity in the other variable gives the same annihilator. Thus $D^{\perp_B}=C$, and nondegeneracy gives $C^{\perp_B}=D$. Since $D\subseteq C$, the restriction of $B$ to $C$ has radical $D$.

Consequently, $B$ induces a perfect pairing
\[
\Phi:\mathfrak g/C\times D\longrightarrow k.
\]
The tensor
\[
T(\bar x,\bar y,\bar z)
=B(x,[y,z])
=B([x,y],z)
\]
is alternating: these two expressions show that it vanishes whenever either adjacent pair of arguments agrees. This proves the necessity in (i).

Conversely, suppose that such a pairing $\Phi$ exists. Choose vector-space complements
\[
C=D\oplus A,\qquad \mathfrak g=W\oplus C,
\]
and identify $W$ with $V$. Then $\mathfrak h=W\oplus D$ is an ideal, and
\[
\mathfrak g=\mathfrak h\oplus A
\]
is a direct sum of ideals. Define
\[
B_{\mathfrak h}(x+d,y+e)
=\Phi(x,e)+\Phi(y,d).
\]
This form is symmetric and nondegenerate. Alternation of $T$ gives its invariance, and
\[
D^{\perp_{B_{\mathfrak h}}}=D.
\]
Thus $\mathfrak h$ is filtered selfdual. Extending $B_{\mathfrak h}$ by any nondegenerate symmetric bilinear form on the central ideal $A$ proves (i) and the asserted decomposition.

For a selfdual algebra, a symmetric invariant form is therefore available. The class-two filtration condition is precisely
\[
D^{\perp_B}=D.
\]
Since $D^{\perp_B}=C$, this proves the first assertion of (ii).

Finally, a perfect pairing $\widetilde\Phi$ as in (ii) forces $C=D$. Indeed, if $r\in C/D$, then $\widetilde\beta(r,u)=0$ for every $u\in U$. Alternation therefore gives
\[
\widetilde\Phi\bigl(r,\widetilde\beta(u,v)\bigr)=0
\]
for all $u,v$. Surjectivity of $\widetilde\beta$ and perfectness of $\widetilde\Phi$ imply $r=0$. Hence $C=D$, and the remaining equivalence follows from (i).
\end{proof}

\setcounter{thm}{3}
Let $\mathfrak g$ be a nilpotent Lie algebra over $\F_q$ whose nilpotency class is less than $p=\operatorname{char}\F_q$. We call the nilpotent group $N=\exp(\mathfrak g)$ \emph{selfdual} if $\mathfrak g$ is selfdual, and \emph{filtered selfdual} if $\mathfrak g$ is filtered selfdual in the sense of Definition~\ref{def:filtered-selfdual-lie}.

\begin{thm}
\label{thm:selfdual-dimension-classes}
Let $\mathfrak n$ be a finite-dimensional nilpotent Lie algebra over $\F_q$, of nilpotency class less than $p=\operatorname{char}\F_q$, admitting a nondegenerate invariant bilinear form. Let $N=\exp(\mathfrak n)$ be the corresponding finite nilpotent group. Then
\[
\bigl\{(\dim\pi)^2:\pi\in\operatorname{Irr}(N)\bigr\}
=
\bigl\{|C|:C\in\Cl(N)\bigr\}
\]
as multisets. In particular, this conclusion holds when $\mathfrak n$ is filtered selfdual.
\end{thm}

\begin{proof}
Let $B$ be a nondegenerate invariant bilinear form on $\mathfrak n$. The map
\[
x\longmapsto B(x,-)
\]
is an $N$-equivariant isomorphism from $\mathfrak n$ to $\mathfrak n^*$. Consequently, adjoint and coadjoint orbits correspond with their sizes preserved. Exponentiation identifies adjoint orbits with conjugacy classes of $N$, while Theorem~\ref{thm:kirillov} identifies coadjoint orbits with irreducible representations of $N$, the orbit size being the square of the representation dimension. These bijections prove the asserted multiset equality.
\end{proof}

\setcounter{thm}{2}
\begin{example}[A selfdual Lie algebra of arbitrary class]
Let $k$ be a field of arbitrary characteristic, and let $\mathfrak s$ be a centerless perfect Lie algebra over $k$ equipped with a nondegenerate invariant symmetric bilinear form $\kappa$. For example, for $n\geq2$, one may take $\mathfrak s=\mathfrak{sl}_n(k)$ with $\kappa(x,y)=\operatorname{tr}(xy)$ whenever $\operatorname{char}(k)$ does not divide $n$. For any $c\geq2$, define
\[
\mathfrak g_c=\mathfrak s\otimes\bigl(tk[t]/(t^{c+1})\bigr)
=\mathfrak s t\oplus\mathfrak s t^2\oplus\cdots\oplus\mathfrak s t^c,
\]
with bracket
\[
[x\otimes t^r,y\otimes t^s]=
\begin{cases}
[x,y]\otimes t^{r+s},&r+s\leq c,\\
0,&r+s>c.
\end{cases}
\]
This Lie algebra has nilpotency class $c$, and its lower central series is
\[
\gamma_i=\bigoplus_{r=i}^{c}\mathfrak s t^r,\qquad 1\leq i\leq c+1.
\]
Since $\mathfrak s$ is centerless, its upper central series is
\[
\mathfrak z_i=\bigoplus_{r=c+1-i}^{c}\mathfrak s t^r.
\]
Therefore $\mathfrak z_i=\gamma_{c+1-i}$ for every $0\leq i\leq c$.

Define a bilinear pairing by
\[
B(x\otimes t^r,y\otimes t^s)=
\begin{cases}
\kappa(x,y),&r+s=c+1,\\
0,&r+s\ne c+1.
\end{cases}
\]
It is symmetric and nondegenerate, pairing $\mathfrak s t^r$ with $\mathfrak s t^{c+1-r}$. Its invariance follows from that of $\kappa$, namely $\kappa([x,y],z)=\kappa(x,[y,z])$. Moreover, $\gamma_i^{\perp_B}=\gamma_{c+2-i}$. Thus the algebras $\mathfrak g_c$ give examples of filtered selfdual Lie algebras in every nilpotency class.
Therefore by Theorem 3.4, for $c<p$, the corresponding nilpotent groups $G_c$ have matching dimension data and conjugacy classes. We do not know if this continues to be the case for $c\geq p$.
\end{example}
\setcounter{thm}{4}

\section{Angles, dispersion, and the logarithmic condition}
\label{sec:angles}

We will be considering nonzero vectors in $\mathbb R^k$, for varying $k$, with nonnegative real coordinates. The Euclidean inner product and norm are denoted by $\langle\ ,\ \rangle$ and $\|\ \|$. 
We begin by making the following definitions.

\begin{defn}[Asymptotically collinear]
\label{def:asym-collinear} 
Let $a^{(n)},b^{(n)}\in\mathbb R^{\ell_n}\setminus\{0\}$, with $\ell_n\to\infty$. The sequences $a^{(n)}$ and $b^{(n)}$ are said to be \textit{asymptotically collinear} if
\[
\frac{\langle a^{(n)},b^{(n)}\rangle}{\|a^{(n)}\|\|b^{(n)}\|}\longrightarrow1.
\]
\end{defn}

\begin{defn}[Asymptotically constant]
\label{def:asym-constant} 
A sequence $a^{(n)}$ of length $\ell_n\to\infty$ is said to be \textit{asymptotically constant} if
\[
\cos\theta(a^{(n)})=\frac{\sum_i a_i^{(n)}}{\sqrt{\ell_n}\|a^{(n)}\|}\longrightarrow 1.
\]
\end{defn}

\begin{defn}[Asymptotically log constant] 
\label{def:asym-log-constant} 
Let $a^{(n)}$ be a sequence of nonnegative real vectors of length $\ell_n$, and set $S_n=\sum_i a_i^{(n)}$. Suppose that $\sqrt{\ell_n}\|a^{(n)}\|>1$ for all sufficiently large $n$. We call the sequence \textit{asymptotically log constant} if
\[
\frac{\log S_n}{\log(\sqrt{\ell_n}\|a^{(n)}\|)}\longrightarrow 1.
\]
\end{defn}

\begin{remark}
\label{rem:constant-implies-log} 
For vectors with entries at least one and lengths tending to infinity, asymptotic constancy implies asymptotic log constancy: the difference of the two logarithms is $\log\cos\theta\to0$, while the denominator tends to infinity. The converse fails.
\end{remark}

For a sequence of nonnegative real numbers $a=(a_1,\ldots,a_k)$ with $a\ne0$, its \emph{mean} and \emph{variance} are
\[
\mu=\frac1k\sum_i a_i,
\qquad
\sigma^2=\frac1k\sum_i(a_i-\mu)^2.
\]
The quantity $\sigma$ is the \emph{standard deviation}, and the \emph{coefficient of variation} is $\cv=\sigma/\mu$.

\begin{proposition}
\label{prop:statistical-interpretation} (Statistical interpretation). {For a vector $a=(a_1,\ldots,a_k)$ with $a_i\geq0$ and $a\ne0$, let $\mu$ be the mean and $\sigma^2$ be the variance. Then, for $\theta(a)$ the angle between the 
vector $a$ and the constant vector $\mathbf{1}$ of size $k$, we have }
\begin{equation}
\label{eq:4.1}
\cos^2\theta(a)=\frac{\mu^2}{\mu^2+\sigma^2}=\frac1{1+(\sigma/\mu)^2}
= \frac{1}{1 + (\cv)^2}.
\end{equation}
{Consequently, asymptotic constancy is equivalent to the coefficient of variation tending to zero.}
\end{proposition}

\begin{proof}
  The identities $\sum_i a_i=k\mu$ and $\sum_i a_i^2=k(\mu^2+\sigma^2)$ give \eqref{eq:4.1}.
\end{proof}

\begin{lemma}\label{lem:log-criterion} (An elementary logarithmic criterion). 
{Let $a=(a_1,\ldots,a_k)$ be a vector with $a_i\geq0$ and $a\ne0$. If $\|a\|>1$ and $S=\sum_i a_i$, then}
\[
0\leq1-\frac{\log S}{\log(\sqrt{k}\|a\|)}\leq\frac{\log \sqrt{k}}{\log(\sqrt{k}\|a\|)}.
\]
{In particular, $\log k=o(\log\|a\|)$ suffices for asymptotic log constancy.}
\end{lemma}

\begin{proof} Nonnegativity and the Cauchy--Schwarz inequality give $\|a\|\leq S\leq\sqrt{k}\|a\|$. Take logarithms and divide by the positive denominator to obtain the asserted inequality.
\end{proof}

For representation dimensions, $\|d_H\|^2=|H|$. For the vector $c_H$ of conjugacy-class sizes, $\|c_H\|\geq|H|/\sqrt{k(H)}$. Hence
\begin{equation}
\label{eq:4.3}
\log k(H)=o(\log|H|)
\end{equation}
implies log constancy for both vectors. 
In later sections, this will be used to prove the logarithmic assertions for fixed-$q$ general linear groups and symmetric groups without any delicate power-sum asymptotics.

\begin{lemma}\label{lem:common-scale} (Comparison through a common scale). {Suppose, for $k \in \mathbb{N}$, $a_k, b_k \in\mathbb R^{n(k)}$ have norm $M_k$, $n(k) \to \infty$ and $\theta(a_k), \theta(b_k) \to 0$ as $k \to \infty$. 
Then, under any reordering of the entries of $b_k$ relative to $a_k$,}
\[
\frac{\|a_k - b_k\|^2}{M_k^2} \longrightarrow 0.
\]
\end{lemma}

\begin{proof} 
For $u_k=n(k)^{-1/2}(1,\ldots,1)\in\mathbb R^{n(k)}$, we have $\|u_k\|=1$ and
\[
\|a_k/M_k - u_k\|^2=2(1-\cos\theta(a_k)),
\]
and similarly for $b_k$. Apply the triangle inequality. The expressions on the right are unchanged by permutations of coordinates.
\end{proof}

\section{Reductive groups over growing finite fields}
\label{sec:reductive}

We begin this section by estimating the number of conjugacy classes in general connected reductive groups over a finite field $\F_q$, a quantity that plays an important role in this work.

\subsection{Conjugacy classes in connected reductive groups} Fulman--Guralnick \cite[Theorem 1.1(2)]{FG12} establish the class-number asymptotic for connected simple algebraic groups. The following theorem extends their theorem to general connected reductive groups\footnote{R. Guralnick outlined to the second author a proof of their theorem for general reductive groups. However, we give our own self-contained variant of their proof.}. 
Its proof explains the rational-conjugacy issue directly.

\begin{thm}\label{thm:conjugacy-classes} (Estimating conjugacy classes). {For a connected reductive group scheme $G$ over $\Z[1/d]$, $d$ an arbitrary nonzero integer, of absolute rank $r$, as $q\to\infty$ through prime powers with $(q,d)=1$, we have:}
\begin{equation}
\label{eq:5.1}
k(G(\F_q))=q^r+O(q^{r-1}).
\end{equation}
{More precisely, the number $k_{\mathrm{vr}}(G(\F_q))$ of semisimple classes whose centralizer is a maximal torus is $q^r+O(q^{r-1})$, and the number of other semisimple classes and the number of nonsemisimple classes are each $O(q^{r-1})$ for $r>0$.}
\end{thm}

\begin{proof} We begin by recalling some structural facts about reductive groups. For each prime power $q$ with $(q,d)=1$, we also write $G$, abusing notation, for the geometric fiber and $F$ for its Frobenius, so that $G^F=G(\F_q)$. All constants below are uniform in $q$ and in the semisimple element, with the reductive group scheme $G$ fixed.

For semisimple $s\in G(\F_q)$, the component group $A(s)=C_G(s)/C_G(s)^\circ$ has bounded order, and the number of unipotent conjugacy classes in $C_G(s)^F$ is uniformly bounded.

To obtain the component-group bound, conjugate $s$ into a maximal torus $T$ over the algebraic closure and let $W=N_G(T)/T$. The connected centralizer has root system $\{\alpha \mid \alpha(s)=1\}$, so only finitely many root subsystems can occur. Further, its component group is a quotient of a subgroup of $W$, and in particular has order at most $|W|$.

For the unipotent-class bound, recall that connected reductive groups have finitely many geometric unipotent classes, and the component groups of their unipotent centralizers are bounded for fixed root datum. Passing to rational classes splits a stable geometric class into at most the order of its component group, by Lang's theorem. These facts give a uniform bound for unipotent classes in connected centralizers. For a normal subgroup $N$ of a finite group $H$, the general inequality $k(H)\leq k(N)k(H/N)$ has a $p$-element analogue $k_p(H)\leq k_p(N)k_p(H/N)$, where $k_p$ counts classes of $p$-elements; see \cite[Lemma 2.3]{FG12}. Taking $H=C_G(s)^F$ and $N=C_G(s)^{\circ F}$, with $H/N$ of bounded order, gives the uniform bound for unipotent classes in disconnected centralizers.

The root data remain fixed across the fibers of $G$ over $\Z[1/d]$. Thus the bounds are uniform across all fibers under consideration.

We now count semisimple classes. The rank-zero case is the trivial connected reductive group and is immediate. Suppose $r>0$, choose an $F$-stable maximal torus $T$, and let $W=N_G(T)/T$ be the geometric Weyl group, with its induced Frobenius action. Let
\[
X=T/\!/W.
\]
Over $\overline{\F}_q$, the points of $X$ parametrize geometric semisimple conjugacy classes. Thus $X(\F_q)$ parametrizes the $F$-stable geometric semisimple classes. Every such class contains a rational representative, by the Lang--Steinberg theorem; it need not be a single rational conjugacy class.

Call a semisimple element very regular if its centralizer is a maximal torus. In particular, this condition excludes disconnected centralizers. For elements of $G$ belonging to $T$, it is equivalent to having trivial stabilizer under $W$. Indeed the identity component is generated by $T$ and the root subgroups with $\alpha(s)=1$, while any other centralizer component is detected in the normalizer of $T$. Hence
\[
T^{\mathrm{vr}}=T\setminus\bigcup_{1\ne w\in W}T^w
\]
is a nonempty $F$-stable open subset, where $T^w=\{t\in T\mid w(t)=t\}$. (Since the Weyl group action is faithful, each $T^w$ with $w\ne1$ is a proper closed subset of $T$.)

As the quotient map is finite, the image $Y$ of $T\setminus T^{\mathrm{vr}}$ is closed of dimension at most $r-1$, and $X^{\mathrm{vr}}=X\setminus Y$ parametrizes the very regular geometric classes.

To count $X(\F_q)$, count pairs $(w,t)\in W\times T(\overline{\F}_q)$ with $wF(t)=t$. Each $F$-stable $W$-orbit contributes exactly $|W|$ such pairs: for each $t$ in that orbit there are $|\operatorname{Stab}_W(t)|$ possible $w$.
Hence
\[
|X(\F_q)|=\frac1{|W|}\sum_{w\in W}|T^{wF}|.
\]
On the character lattice of $T$, the twisted Frobenius $wF$ acts as $q\Phi_w$, where $\Phi_w$ has finite order. Its eigenvalues have absolute value one. By the well-known formula for the order of a torus, we have:
\[
(q-1)^r\leq |T^{wF}|=|\det(q\Phi_w-I)|\leq(q+1)^r.
\]
Averaging these bounds, and applying the point bound in lower dimension to $Y$, gives
\begin{equation}
\label{eq:5.2}
|X(\F_q)|=q^r+O(q^{r-1}),\qquad |Y(\F_q)|=O(q^{r-1}).
\end{equation}
The bound $|Y(\F_q)|=O(q^{r-1})$ in \eqref{eq:5.2} is the standard upper point bound for the closed subset $Y$, whose dimension is at most $r-1$; see \cite{LW54}. For a fixed model it is uniform because the closed loci have bounded defining complexity. The first estimate is uniform directly from the displayed torus bounds. The estimate $|X(\F_q)|=q^r+O(q^{r-1})$ is sharper than the general Lang--Weil estimate $q^r+O(q^{r-1/2})$ for the geometrically irreducible variety $X=T/\!/W$ of dimension $r$. The sharper error term follows from the explicit description of $X(\F_q)$ by the preceding average of torus orders.

To pass from geometric classes to rational classes, fix a rational semisimple representative $s$. Rational classes in its geometric class are parametrized by
\[
\ker\bigl(H^1(F,C_G(s))\longrightarrow H^1(F,G)\bigr)=H^1(F,C_G(s)),
\]
since $H^1(F,G)=1$. 
Lang's theorem 
on the connected component identifies this pointed set with $H^1(F,A(s))$. Its elements are Frobenius-twisted conjugacy classes in the finite geometric component group $A(s)$. In particular,
\[
|H^1(F,A(s))|\leq|A(s)|\leq|W|.
\]
The bound uses the geometric component group, not its set of $F$-fixed elements. On $X^{\mathrm{vr}}(\F_q)$ the component group is trivial, so there is exactly one rational conjugacy class over each point. Over $Y(\F_q)$ the fiber is uniformly bounded. It follows (using \eqref{eq:5.2}) that
\[
k_{\mathrm{vr}}(G(\F_q))=q^r+O(q^{r-1}),\qquad
k_{\mathrm{ss}}(G(\F_q))-k_{\mathrm{vr}}(G(\F_q))=O(q^{r-1}).
\]
It remains to count nonsemisimple classes. Write a rational element as its Jordan decomposition $g=su=us$; both parts are rational. For a fixed rational semisimple class represented by $s$, the classes of elements having that semisimple part correspond exactly to unipotent conjugacy classes in $C_G(s)^F$. If $s$ is very regular, its centralizer is a torus and has no nonidentity unipotent elements. All nonsemisimple classes therefore lie over the $O(q^{r-1})$ exceptional semisimple classes. The uniform unipotent-class bound established at the beginning of the proof gives
\[
k_{\mathrm{ns}}(G(\F_q))=O(q^{r-1}).
\]
Combining these estimates proves \eqref{eq:5.1} and the two finer bounds, uniformly as $q\to\infty$ through prime powers with $(q,d)=1$.
\end{proof}

The estimate for the number of very regular classes will also be used below in Theorem~\ref{thm:class-size-sum}.

\subsection{Dimension concentration and the Gelfand--Graev lower bound} Let $G$ be a connected reductive group over $\F_p$, and let $q=p^f$. Let $D=\dim G$, $r$ be its absolute rank (including the central torus), and let $N=(D-r)/2$. Let $F$ be the Frobenius automorphism defining $G(\F_q)$, so that $G^F=G(\F_q)$. Choose an $F$-stable Borel subgroup $B$ with unipotent radical $U$. Then $\dim U=N$ and $|U^F|=q^N$. Choose a regular character $\psi$ of $U^F$ (a Whittaker datum), and let
\[
\Gamma_\psi=\operatorname{Ind}_{U^F}^{G^F}\psi.
\]
This representation of $G^F$, called the Gelfand--Graev representation, is multiplicity-free for connected reductive groups; connectedness of the center is not required; see \cite[Theorem 12.3.4]{DiM20}. Then, as $\Gamma_\psi$ is multiplicity-free,
\begin{equation}
\label{eq:5.4}
S_q:=\sum_{\pi\in\Irr(G^F)}d_\pi\geq\dim\Gamma_\psi=\frac{|G^F|}{q^N}.
\end{equation}

\begin{thm}
\label{thm:dimension-sum-reductive} {Let $G$ be a connected reductive group over $\F_p$, of dimension $D$ and positive absolute rank $r$. Put $q=p^f$ and $N=(D-r)/2$. Then, as $f\to\infty$,}
\begin{equation}
\label{eq:5.5}
\sum_{\pi\in\Irr(G(\F_q))}d_\pi=q^{N+r}(1+O(q^{-1})),\qquad
1-\cos\theta(d_{G(\F_q)})=O(q^{-1}).
\end{equation}
{More precisely,}
\[
\frac1{k(G(\F_q))}\sum_\pi\left(\frac{d_\pi}{q^N}-1\right)^2=O(q^{-1}).
\]
{The same estimates hold through the finite-field fibers of a fixed reductive model in Theorem~\ref{thm:conjugacy-classes}. The implied constants may depend on the fixed group or model.}
\end{thm}

\begin{proof} Write $H=G(\F_q)$ and $k=k(H)$. The formula for the order of a reductive group and the class-count estimate give
\[
|H|=q^D(1+O(q^{-1})),\qquad k=q^r(1+O(q^{-1})).
\]
For a fixed group over $\F_p$, the proof of Theorem~\ref{thm:conjugacy-classes} gives the same class-count estimate as $f\to\infty$: the root datum and defining complexity remain fixed. By \eqref{eq:5.4} and the Cauchy--Schwarz inequality,
\[
\frac{|H|}{q^N}\leq S_q=\sum_{\pi\in\Irr(H)}d_\pi\leq\sqrt{k|H|}.
\]
Both bounds equal $q^{N+r}(1+O(q^{-1}))$. Dividing by $\sqrt{k|H|}$ proves the angle estimate. 

For the last assertion, expand the square:
\[
\frac1k\sum_\pi\left(\frac{d_\pi}{q^N}-1\right)^2
=\frac{|H|}{kq^{2N}}-\frac{2S_q}{kq^N}+1=O(q^{-1}).\qedhere
\]
\end{proof}

\begin{remark}
\label{rem:proportion-irreducibles}
For each fixed $\epsilon>0$, the proportion of irreducibles with $|d_\pi/q^N-1|>\epsilon$ is therefore $O(\epsilon^{-2}q^{-1})$. This is the precise meaning of ``most irreducible representations have dimension roughly $q^N$.''
\end{remark}

Presumably the following corollary is well known, though we have not been able to find a reference for it.

\begin{corollary}[Dominance of a chosen generic family]
\label{cor:generic-family-dominance}
Let $\Irr_\psi(G^F)$ be the irreducible constituents of the chosen $\Gamma_\psi$. Then the total dimension of the irreducibles outside this family satisfies
\[
\sum_{\pi\in\Irr(G^F)\setminus\Irr_\psi(G^F)}d_\pi=O(q^{N+r-1}).
\]
Equivalently,
\[
\frac{\sum_{\pi\in\Irr(G^F)}d_\pi}{\sum_{\pi\in\Irr_\psi(G^F)}d_\pi}=1+O(q^{-1}).
\]
\end{corollary}

\begin{proof}
Since $\Gamma_\psi$ is multiplicity-free, the order formula gives
\[
\sum_{\pi\in\Irr_\psi(G^F)}d_\pi
=\dim\Gamma_\psi=\frac{|G^F|}{q^N}
=q^{N+r}+O(q^{N+r-1}).
\]
By \eqref{eq:5.5}, the sum over all irreducibles has the same expansion. Subtraction proves the first assertion; division by the displayed sum gives the ratio formulation.
\end{proof}

\subsection{Comparison with square roots of class sizes}

\begin{thm}
\label{thm:class-size-sum} {In the setting of Theorem~\ref{thm:dimension-sum-reductive},}
\begin{equation}
\label{eq:class-size-sum-asymptotic}
\sum_{C\in\Cl(G(\F_q))}\sqrt{|C|}=q^{N+r}(1+O(q^{-1})),\qquad
1-\cos\theta(s_{G(\F_q)})=O(q^{-1}).
\end{equation}
{Moreover,}
\[
\frac1{k(G(\F_q))}\sum_C\left(\frac{\sqrt{|C|}}{q^N}-1\right)^2=O(q^{-1}).
\]
{For every bijective pairing between irreducible representations and conjugacy classes,}
\begin{equation}
\label{eq:5.7}
\frac{\|d_{G(\F_q)}-s_{G(\F_q)}\|^2}{|G(\F_q)|}=O(q^{-1}),
\end{equation}
{uniformly over the pairings, with the group $G$ fixed.}
\end{thm}

\begin{proof} Write $H=G(\F_q)$ and $k=k(H)$. For a very regular semisimple element $s$, its centralizer $T_s$ is a maximal torus of dimension $r$. The well-known formula for the order of a torus gives, uniformly,
\[
(q-1)^r\leq|T_s(\F_q)|\leq(q+1)^r.
\]
Consequently, for the conjugacy class $s^H$ in $H$ passing through $s$,
\[
\sqrt{|s^H|}=\sqrt{\frac{|H|}{|T_s(\F_q)|}}=q^N(1+O(q^{-1})).
\]
There are $q^r+O(q^{r-1})$ very regular semisimple conjugacy classes by Theorem~\ref{thm:conjugacy-classes}. Their contribution gives a lower bound $q^{N+r}(1+O(q^{-1}))$ for $T_q:=\sum_C\sqrt{|C|}$. The Cauchy--Schwarz inequality gives the upper bound $T_q\leq\sqrt{k|H|}$. This proves \eqref{eq:class-size-sum-asymptotic}. Expanding the square gives
\[
\frac1k\sum_C\left(\frac{\sqrt{|C|}}{q^N}-1\right)^2
=\frac{|H|}{kq^{2N}}-\frac{2T_q}{kq^N}+1=O(q^{-1}).
\]
Now choose any bijective pairing $\pi\mapsto C_\pi$. For each $\pi$, write
\[
d_\pi-\sqrt{|C_\pi|}
=(d_\pi-q^N)-(\sqrt{|C_\pi|}-q^N).
\]
The elementary inequality $(u-v)^2\leq 2u^2+2v^2$ therefore gives
\[
(d_\pi-\sqrt{|C_\pi|})^2
\leq 2(d_\pi-q^N)^2+2(\sqrt{|C_\pi|}-q^N)^2.
\]
Summing over $\pi$, and using that the pairing runs through each conjugacy class exactly once, yields
\begin{equation}\label{eq:pairing-summed-bound}
\sum_\pi(d_\pi-\sqrt{|C_\pi|})^2
\leq 2\sum_\pi(d_\pi-q^N)^2
+2\sum_C(\sqrt{|C|}-q^N)^2.
\end{equation}
The mean-square estimate in Theorem~\ref{thm:dimension-sum-reductive} and the estimate just proved say, respectively, that
\begin{align*}
\sum_\pi(d_\pi-q^N)^2
&=kq^{2N}\left[\frac1k\sum_\pi
\left(\frac{d_\pi}{q^N}-1\right)^2\right]
=O(kq^{2N-1}),\\
\sum_C(\sqrt{|C|}-q^N)^2
&=kq^{2N}\left[\frac1k\sum_C
\left(\frac{\sqrt{|C|}}{q^N}-1\right)^2\right]
=O(kq^{2N-1}).
\end{align*}
Since $2N+r=D$, the order estimate $|H|=q^D(1+O(q^{-1}))$ and the class-count estimate $k=q^r(1+O(q^{-1}))$ from Theorem~\ref{thm:conjugacy-classes} give
\[
\frac{kq^{2N}}{|H|}=1+O(q^{-1}).
\]
Dividing \eqref{eq:pairing-summed-bound} by $|H|$ now gives
\[
\frac{\|d_H-s_H\|^2}{|H|}
=O\left(\frac{kq^{2N-1}}{|H|}\right)=O(q^{-1}),
\]
which proves \eqref{eq:5.7}. The two sums on the right are independent of the chosen pairing, so the implied constant is uniform over all pairings.
\end{proof}

\subsection{Relation to earlier dimension-sum bounds} Kowalski \cite[Proposition 3]{Kow11} proves that for split connected reductive $G$ with connected center,
\[
\sum_{\pi\in\operatorname{Irr}(G(\F_q))}d_\pi\leq(q+1)^{(D+r)/2}\left(1+\frac{2r|W|}{q-1}\right).
\]
This also appears in \cite[Proposition 5.5]{Kow08}. Vinroot \cite[Proposition 7.1]{Vin10} obtains, for connected reductive groups with connected center,
\[
\sum_{\pi\in\Irr(G(\F_q))}d_\pi\geq q^{(D-r)/2}(q-1)^r.
\]
For connected classical groups with connected center in odd characteristic, \cite[Theorem 6.1]{Vin10} gives the upper bound $(q+1)^{(D+r)/2}$, without requiring splitness. These bounds already imply the dimension-sum asymptotic for various reductive groups, though \textit{always with connected center}, which we do not assume in this work.

\section{The general linear group $\GL_n(\F_q)$, $n$ varying, $q$ fixed}
\label{sec:gl}
Fix a prime power $q$, and now let $n\to\infty$ in $\GL_n(\F_q)$. Write
\[
B_n(q)=\sum_{\pi\in\operatorname{Irr}(\GL_n(\F_q))}d_\pi,\quad C_n(q)=k(\GL_n(\F_q)),\quad D_n(q)=|\GL_n(\F_q)|.
\]
The following known formulas lead to an exact limiting angle.

\begin{thm}\label{thm:dimension-sum-formula} (Dimension-sum formula). {The sum $B_n(q)$ equals the number of invertible symmetric $n$-by-$n$ matrices over $\F_q$. Explicitly,}
\begin{align*}
B_{2m+1}(q)&=q^{m^2+m}\prod_{j=0}^m(q^{2j+1}-1),\\
B_{2m}(q)&=q^{m^2+m}\prod_{j=1}^m(q^{2j-1}-1).
\end{align*}
This theorem is due to Gow \cite{Gow83} for $q$ odd, and to Klyachko and Macdonald for all $q$, as discussed in \cite[Introduction]{Vin10}.
\end{thm}

\begin{thm}[{Feit--Fine~\cite{FF60}, Macdonald~\cite[Eq. (1.13)]{Mac81}, Fulman--Guralnick~\cite[Proposition 3.5(i)]{FG12}}]
\label{thm:feit-fine-macdonald} 
{With $C_0(q)=1$,}
\[
\sum_{n\geq0}C_n(q)t^n=\prod_{j\geq1}\frac{1-t^j}{1-qt^j}.
\]
{In particular, $C_n(q)$ is a polynomial in $q$ of degree $n$ and leading coefficient one. For fixed $q$,}
\[
\lim_{n\to\infty}\frac{C_n(q)}{q^n}=1.
\]
\end{thm}

The second assertion can also be read from the simple pole at $t=1/q$: after separating the factor $(1-qt)^{-1}$, the remaining product takes value one there by telescoping.
Recall that $\gamma(q)=\sum_{j\geq0}q^{-j(j+1)/2}$.

\begin{thm}
\label{thm:gl-angle} {For fixed prime power $q$,}
\begin{enumerate}[label=(\alph*)]
\item {The dimension vector satisfies}
\[
\lim_{n\to\infty}\cos^2\theta(d_{\GL_n(\F_q)})=\frac1{\gamma(q)} < 1.
\]
\item {The dimension vector is asymptotically log constant.}
\end{enumerate}
\end{thm}

\begin{proof} Put $\widehat C_n(q)=C_n(q)/q^n$. The group order is
\[
D_n(q)=q^{n^2}\prod_{i=1}^n(1-q^{-i}).
\]
The formulas in Theorem~\ref{thm:dimension-sum-formula} yield
\begin{align*}
\frac{B_{2m}(q)^2}{C_{2m}(q)D_{2m}(q)}&=\frac1{\widehat C_{2m}(q)}\frac{\prod_{j=1}^m(1-q^{-(2j-1)})}{\prod_{j=1}^m(1-q^{-2j})},\\
\frac{B_{2m+1}(q)^2}{C_{2m+1}(q)D_{2m+1}(q)}&=\frac1{\widehat C_{2m+1}(q)}\frac{\prod_{j=0}^m(1-q^{-(2j+1)})}{\prod_{j=1}^m(1-q^{-2j})}.
\end{align*}
Since $\widehat C_n(q)\to1$, both subsequences tend to
\[
\prod_{j\geq1}\frac{1-q^{-(2j-1)}}{1-q^{-2j}}.
\]
Gauss's identity \cite[p.~23]{And76},
\[
\sum_{j\geq0}t^{j(j+1)/2}=\prod_{j\geq1}\frac{1-t^{2j}}{1-t^{2j-1}},\qquad |t|<1,
\]
identifies this product with $1/\gamma(q)$. This is the squared cosine by \eqref{eq:1.1}. For (b), $\log C_n(q)=n\log q+o(1)$, while $\log D_n(q)=n^2\log q+O(1)$. Now apply \eqref{eq:4.3} to complete the proof.
\end{proof}

\begin{remark}
\label{rem:gl-limiting-angle} The limiting angle is $\arccos(\gamma(q)^{-1/2})\in(0,\pi/2)$, so the dimension vector is not asymptotically constant at fixed $q$. In terms of Proposition~\ref{prop:statistical-interpretation}, the squared coefficient of variation tends to $\gamma(q)-1$. As a formal analytic comparison, $\gamma(q)\to\infty$ as the real parameter $q\downarrow1$, and the angle tends to $\pi/2$. This is not a limit through finite fields. Conversely $\gamma(q)\to1$ as $q\to\infty$.
\end{remark}

Analogous calculations for unitary, symplectic, orthogonal, and special linear groups require their own dimension and class-count inputs. We do not discuss these groups here.

\section{Symmetric groups: dimension statistics}
\label{sec:dimensions}

Let $S_n$ be the symmetric group on $n$ letters.
Let $\mathcal P(n)$ be the set of partitions $\lambda=(\lambda_1\geq\lambda_2\geq\cdots)$ of $n$, and let $p(n)=|\mathcal P(n)|$. Both irreducible representations and conjugacy classes of $S_n$ are indexed by $\mathcal P(n)$. Write $d_\lambda$ for the irreducible dimension and $c_\lambda$ for the conjugacy-class size. 
Conjugation of partitions preserves dimensions and corresponds to twisting by sign.
The uniform measure on $\mathcal P(n)$ assigns mass $1/p(n)$ to each partition, whereas Plancherel measure assigns mass $d_\lambda^2/n!$ to $\lambda$. All proportions of partitions considered in this section use the uniform measure.

The well-known hook-length formula computes individual dimensions but leaves asymptotic questions open. Here is a well-known question in the subject.
\begin{qn}
\label{qn:max-dimension} Can one describe the exact maximum $m_n=\max_{\lambda\vdash n}d_\lambda$ and its maximizing partitions, including possible ties other than conjugate partitions? More generally, what is the statistical distribution of $d_\lambda$ under each of these two measures?
\end{qn}

The limit-shape results of Vershik--Kerov~\cite{VK77} and Logan--Shepp~\cite{LS77}, and the entropy theorem of Bufetov~\cite{Buf12}, concern Plancherel measure and provide information about the statistical distribution of $d_\lambda$ under that measure. Here we use the zeroth, first, and second moments of $d_\lambda$, together with the Cauchy--Schwarz inequality, to show that the dimension vector is not asymptotically constant, but is asymptotically log constant. Corollaries~\ref{cor:uniform-threshold} and~\ref{cor:threshold-proportion} show that the proportion of irreducibles whose dimensions exceed any fixed positive fraction of the maximal dimension tends to zero.

Note that
\begin{align*}
\sum_{\lambda\vdash n}1 &= p(n),\\
\sum_{\lambda\vdash n}d_\lambda &= I(n),\\
\sum_{\lambda\vdash n}d_\lambda^2 &= n!,
\end{align*}
where $I(n)=|\{\sigma \in S_n \mid \sigma^2=1\}|$ is the number of involutions (including the identity). The middle equality is the Frobenius--Schur involution formula, since every irreducible representation of $S_n$ is realizable over $\R$ (indeed over $\Q$). 
Counting disjoint transpositions gives
\[
I(n)=\sum_{0\leq2j\leq n}\frac{n!}{2^j j!(n-2j)!}.
\]

\begin{thm}
\label{thm:sn-dimension-angle}
\begin{enumerate}[label=(\alph*)]
\item {The dimension vectors of $S_n$ are asymptotically log constant as $n\to\infty$.}
\item {Put}
\[
a_0=\pi\sqrt{2/3}-2>0,\qquad K=\sqrt{\frac6{\pi e}},\qquad R_n=\frac{I(n)^2}{n!\,p(n)}.
\]
{Then}
\begin{equation}
\label{eq:8.2}
\cos^2\theta(d_{S_n})=R_n=\frac{I(n)^2}{n!\,p(n)}\sim K\sqrt n\,e^{-a_0\sqrt n}.
\end{equation}
{In particular, $\theta(d_{S_n})\to\pi/2$; hence the dimension vectors are not asymptotically constant.}
\end{enumerate}
\end{thm}

Thus, the dimension vector is log constant, as in both reductive-group regimes (Theorem~\ref{thm:dimension-sum-reductive} together with Remark~\ref{rem:constant-implies-log}, and Theorem~\ref{thm:gl-angle}(b)), but its angle now tends to $\pi/2$. At fixed $q$, the general-linear angle instead tends to a strictly intermediate value by Theorem~\ref{thm:gl-angle}(a).

\begin{proof} The Hardy--Ramanujan formula~\cite{And76}, the involution asymptotic of Chowla--Herstein--Moore~\cite{CHM51} and Moser--Wyman~\cite{MW55}, and Stirling's formula give
\begin{align}
p(n)&\sim\frac1{4\sqrt3\,n}\exp\bigl(\pi\sqrt{2n/3}\bigr),\label{eq:8.3}\\
I(n)&\sim\frac1{\sqrt2\,e^{1/4}}(n/e)^{n/2}e^{\sqrt n},\label{eq:8.4}\\
n!&\sim\sqrt{2\pi n}(n/e)^n.\label{eq:8.5}
\end{align}
Since $\log p(n)=O(\sqrt n)=o(\log n!)$, Lemma~\ref{lem:log-criterion} gives (a).

Direct substitution into $R_n$ proves (b), including its constant, as we now see. The relevant logarithmic expansions, with constants retained, are
\begin{align*}
\log I(n)^2&=n\log n-n+2\sqrt n-\log2-\tfrac12+o(1),\\
\log(n!\,p(n))&=n\log n-n+\pi\sqrt{2n/3}-\tfrac12\log n
+\log\frac{\sqrt{2\pi}}{4\sqrt3}+o(1).
\end{align*}
Their difference is
\[
\log R_n=-a_0\sqrt n+\tfrac12\log n+\log K+o(1)\longrightarrow-\infty.
\]
This proves \eqref{eq:8.2}.
\end{proof}

\begin{remark}
Equations~\eqref{eq:8.3}--\eqref{eq:8.5} in the proof of Theorem~\ref{thm:sn-dimension-angle} show explicitly that
\[
\frac{\log I(n)^2}{\log(n!\,p(n))}\longrightarrow1,
\]

\noindent whereas
\[
\log I(n)^2-\log(n!\,p(n))=\log R_n\longrightarrow-\infty.
\]
This illustrates concretely the weakness of log constancy. The formal $q\downarrow1$ comparison after Theorem~\ref{thm:gl-angle} concerns the angle conclusion in part (b) of Theorem~\ref{thm:sn-dimension-angle}.
\end{remark}

\begin{corollary}\label{cor:uniform-threshold}
{Let $m_n=\max_{\lambda\vdash n}d_\lambda$. Choose any $0<A_n\leq1$, and set}
\[
C_n=\frac{|\{\lambda\vdash n \mid d_\lambda\geq A_nm_n\}|}{p(n)}.
\]
{For every $n$,}
\begin{equation}
\label{eq:8.6}
A_n C_n\leq R_n.
\end{equation}
{Consequently, for any such sequence $A_n$, and for all sufficiently large $n$,}
\[
(A_nC_n)^2\leq e^{-a_0\sqrt n}.
\]
\end{corollary}

\begin{proof} Because $d_\lambda\leq m_n$,
\[
n!=\sum_\lambda d_\lambda^2\leq m_n\sum_\lambda d_\lambda=m_nI(n), 
\]
giving $m_n \geq n!/I(n)$.
The defining set of partitions gives $m_nA_nC_np(n)\leq I(n)$. Combining these inequalities proves \eqref{eq:8.6}.

Finally, by Theorem~\ref{thm:sn-dimension-angle}(b),
\[
R_n^2\sim K^2n e^{-2a_0\sqrt n}=o(e^{-a_0\sqrt n}),
\]
which proves the last assertion.
\end{proof}

\begin{corollary}
\label{cor:threshold-proportion} {For fixed $0<A\leq1$, the proportion of partitions with $d_\lambda\geq Am_n$ satisfies}
\[
C_n\leq\frac{R_n}{A}\sim\frac KA\sqrt n\,e^{-a_0\sqrt n}\longrightarrow0.
\]
Thus, the proportion of partitions with $d_\lambda\geq Am_n$ goes to zero.
\end{corollary}

\section{Symmetric groups: statistics of conjugacy-class sizes}
\label{sec:classes}
This section concerns $c_{S_n}=(c_\lambda)_{\lambda\vdash n}$. Using the zeroth, first, and second moments, we show that conjugacy-class sizes are not asymptotically constant.
The zeroth and first power sums are $p(n)$ and $n!$. For the second power sum, put
\[
W_n=\sum_{\lambda\vdash n}c_\lambda^2,\qquad Q_n=\frac{W_n}{(n!)^2}=\sum_{\lambda\vdash n}\frac1{z_\lambda^2}.
\]
Here $z_\lambda=n!/c_\lambda$.

\begin{proposition}[{Flajolet--Fusy--Gourdon--Panario--Pouyanne~\cite[Proposition 4]{FFGPP06}}]
\label{prop:ffgpp} 
{There is a constant $\kappa\approx4.263403514$ such that}
\[
Q_n\sim\frac\kappa{n^2},\qquad W_n\sim\kappa\frac{(n!)^2}{n^2}.
\]
\end{proposition}

\begin{thm}
\label{thm:class-size-angle}\quad{(a) The vector of conjugacy-class sizes of $S_n$ is asymptotically log constant. (b) Its squared cosine has the asymptotic}
\begin{equation}
\label{eq:9.2}
\cos^2\theta(c_{S_n})=\frac{(n!)^2}{p(n)W_n}\sim\frac{n^2}{\kappa p(n)}
\sim\frac{4\sqrt3}{\kappa}n^3e^{-\pi\sqrt{2n/3}}.
\end{equation}
{In particular the angle tends to $\pi/2$, not to zero.}
\end{thm}

\begin{proof} Part (a) follows from $\log p(n)=o(\log n!)$ and \eqref{eq:4.3} for the vector of conjugacy-class sizes.

For (b), the class equation gives $\sum c_\lambda=n!$, and Proposition~\ref{prop:ffgpp} evaluates $W_n$. Then apply \eqref{eq:8.3}, the Hardy--Ramanujan asymptotic formula for $p(n)$.
\end{proof}

\begin{remark}
The dimension vector and the vector of conjugacy-class sizes both have angles tending to $\pi/2$, but at different rates.
\end{remark} 

\section*{AI statement}

ChatGPT was used to assist with proofreading, improving the exposition, and identifying potentially relevant references. The authors independently verified the mathematical arguments, references, and final content of the paper.

\section*{Acknowledgments}

We thank U.K. Anandavardhanan, Purvi Gupta, Manoj Yadav, Dan Romik, R. Guralnick and S. Velmurugan
for discussions on the themes arising in this work.
We thank ICTS for bringing the authors together to work on this problem.
The first author (AA) acknowledges support from the DST FIST program - 2021 [TPN - 700661].

\newcommand{\etalchar}[1]{$^{#1}$}
\providecommand{\bysame}{\leavevmode\hbox to3em{\hrulefill}\thinspace}
\providecommand{\MR}{\relax\ifhmode\unskip\space\fi MR }

\end{document}